\documentclass[12pt]{amsart}
\usepackage[utf8]{inputenc}

\usepackage{amsmath}%
\usepackage{amsfonts}%
\usepackage{amssymb}%
\usepackage{amsthm}
\usepackage{bm}
\usepackage{graphicx}
\usepackage{stmaryrd}
\usepackage{centernot}
\usepackage{url}
\usepackage[english]{babel}
\usepackage[margin=1in]{geometry}

\theoremstyle{plain}
\newtheorem{theorem}{Theorem}[section]
\newtheorem{corollary}[theorem]{Corollary}
\newtheorem{lemma}[theorem]{Lemma}
\newtheorem{fact}[theorem]{Fact}
\newtheorem{proposition}[theorem]{Proposition}

\newtheorem{claim}{Claim}
\theoremstyle{definition}
\newtheorem{definition}[theorem]{Definition}

\newtheorem{question}[theorem]{Question}

\newcommand{\R}{\mathbb{R}}
\newcommand{\supp}{\text{supp}}

\newcommand{\BPi}{\mathbf{\Pi}}
\newcommand{\BSigma}{\mathbf{\Sigma}}
\newcommand{\Lab}{\mathrm{Lab}}

\begin{document}

\title{Countable Borel treeable equivalence relations are classifiable by $\ell_1$}
\date{\today}
\author{Shaun Allison}
\address{Department of Mathematics, University of Toronto, Toronto, Canada}
\email{shaunpallison@gmail.com}
\urladdr{https://www.shaunallison.com}



\begin{abstract}
    In \cite{GaoJackson2015}, it was shown that any countable Borel equivalence relation (CBER) induced by a countable abelian Polish group is hyperfinite.
    This prompted Hjorth to ask if this is in fact true for all CBERs classifiable by (uncountable) abelian Polish groups.
    
    We describe reductions involving \emph{free Banach spaces} to show that every treeable CBER is classifiable by an abelian Polish group.
    As there exist treeable CBERs that are not hyperfinite, this answers Hjorth's question in the negative.
    
    On the other hand, we show that any CBER classifiable by a countable product of locally compact abelian Polish groups (such as $\mathbb{R}^\omega$) is indeed hyperfinite.
    We use a small fragment of the \emph{Hjorth analysis} of Polish group actions, which is Hjorth's generalization of the Scott analysis of countable structures to Polish group actions.
\end{abstract}

\maketitle

\section{Introduction}

We say that a Borel equivalence relation $E$ is \textbf{classifiable by a Polish group $G$} if there is a Borel reduction from $E$ to an orbit equivalence relation induced by a continuous action of $G$ on a Polish space.
A \textbf{countable equivalence relation} is an equivalence relation in which every class is countable.

In \cite{GaoJackson2015}, Gao-Jackson proved that every countable equivalence relation classifiable by a countable abelian group is hyperfinite using deep combinatorial arguments.
This was originally proved for $\mathbb{Z}^n$ by Weiss.
The Gao-Jackson result prompted Hjorth to ask if this result can be strengthened as follows.

\begin{question}[Hjorth, \cite{HjorthQuestion}]\label{question:hjorth}
Suppose $E$ is a countable equivalence relation which is classifiable by an abelian Polish group. Must $E$ be hyperfinite?
\end{question}

There is good reason to think this might be true. For example, in \cite{DingGao2017}, Ding-Gao show how the Gao-Jackson result can be easily lifted to the case of classifiability by a non-Archimedean abelian Polish group, and in \cite{Cotton2019}, Cotton also gets a positive answer in the case of classifiability by a locally-compact abelian Polish group.

However, in this paper, we prove the following:

\begin{theorem}\label{thm:treeable}
Every countable treeable equivalence relation is classifiable by an abelian Polish group. In particular, they are classifiable by $\ell_1$.
\end{theorem}

We know that there exists countable Borel treeable equivalence relations which are not hyperfinite, such as the free part of the action of the free group $F_2$ on $2^{F_2}$ by shifts (see \cite{JKL2002}). This allows us to answer Hjorth's question in the negative.

Hjorth also asked Question \ref{question:hjorth} for the specific case of equivalence relations classifiable by $\mathbb{R}^\omega$. For this we prove the following:

\begin{theorem}\label{thm:lca_classifiable}
Every $\BPi^0_3$ equivalence relation which is classifiable by a countable product of locally compact abelian Polish groups is Borel-reducible to $E_0^\omega$.
\end{theorem}

Essentially countable equivalence relations are potentially $\BPi^0_3$ (in fact potentially $\BSigma^0_2$), thus by the Hjorth-Kechris Sixth Dichotomy Theorem, we can conclude the following:

\begin{corollary}
The answer to Question \ref{question:hjorth} is yes for the special case of equivalence relations classifiable by countable products of locally compact abelian Polish groups, such as $\R^\omega$.
\end{corollary}

We prove Theorem \ref{thm:treeable} in Section \ref{sec:treeable} and Theorem \ref{thm:lca_classifiable} in Section \ref{sec:lca_classifiable}. We make some more remarks and leave some further questions at the end.

\subsection{Acknowledgements}

The author would like to thank Clinton Conley, Aristotelis Panagiotopoulos, Su Gao, Omer Ben-Neria, Slawek Solecki, and Benjamin Weiss for many valuable conversations. 

An earlier version of the first half of this paper appeared in the author's PhD thesis.

\section{Preliminaries}

We assume a basic understanding of the theory of Borel reducibility, e.g. \cite{Gao2008}. 
We say that a Polish group $G$ \textbf{involves} a Polish group $H$ if there is a closed subgroup $G'$ and a continuous surjective homomorphism from $G'$ onto $H$.
The following fact is important in the classification theory of equivalence relations.
\begin{fact}[Mackey, Hjorth]\label{fact:mackey-hjorth}
Suppose $G$ and $H$ are Polish groups such that $G$ involves $H$. Then for any Polish $H$-space $Y$ there is a Polish $G$-space $X$ such that $E^G_X$ and $E^H_Y$ are Borel bi-reducible.
\end{fact}
In particular, if an equivalence relation is classifiable by $H$ and $G$ involves $H$, then by the previous fact, it is classifiable by $G$.
See \cite[Theorem 2.3.5]{BeckerKechris1996} for its original, stronger formulation as well as its proof.

Let $X$ be a Polish space.
Recall that $\mathcal{F}(X)$ denotes the standard Borel space of all closed subsets of $X$ with the Effros Borel structure, namely the $\sigma$-algebra generated by sets of the form 
\[\{C \in \mathcal{F}(X) \mid C \cap U \neq \emptyset\}\]
for open $U \subseteq X$.
If $X$ is a Polish $G$-space, meaning that it comes with a continuous action $G \curvearrowright X$ of a Polish group $G$, then the action
\[ G \curvearrowright \mathcal{F}, \quad g \cdot C = \{g \cdot x \mid x \in C\}\]
is Borel, making it a Borel $G$-space (i.e. a standard Borel space equipped with Borel action of a Polish group).
Recall that every Borel $G$-space can be equipped with a Polish topology generating the same Borel sets making it a Polish $G$-space \cite[Therorem 5.2.1]{BeckerKechris1996}.

Given a Polish space $X$, we write $X^{\le \omega}$ to refer to the Polish space $\bigsqcup_{\alpha \le \omega} X^\alpha$.
The following is easy to check.
\begin{lemma}\label{lem:effros_borel}
The map
\[ X^{\le \omega} \rightarrow \mathcal{F}(X), \quad (x_i)_{i <\alpha} \mapsto \overline{\{x_i \mid i < \alpha\}}\]
is Borel.
\end{lemma}

\subsection{Free Banach spaces}
Given a metric space $(X, d)$ with $d \le 1$ one can define the \textbf{free Banach space over X}. 
The construction was first introduced by Pestov in \cite{Pestov1986}.
We will review the most relevant aspects as presented in \cite{GaoPestov2003}.

We consider the metric space $(X \sqcup \{*\}, d)$ where $*$ is some new point such that $d(x, *) = d(*, x) = 1$ for every $x \in X$. 
Define $L(X)$ to be the unique real vector space which has $X \sqcup \{*\}$ as a Hamel basis.

For $x \in L(X)$, we define the norm $\|x\|$ to be the infimum over all sums
\[\sum_{i=1}^n |\lambda_i| d(x_i, y_i)\]
for $n \in \omega$, $x_i, y_i \in X \sqcup \{*\}$, and $\lambda_i \in \mathbb{R}$ satisfying
\[ \sum_{i=1}^n \lambda_i x_i = x \quad \text{and} \quad  \sum_{i=1}^n \lambda_i y_i \in \text{Span}\{*\}.\]

To see that this is indeed a norm on $L(X)$, see e.g. \cite[Theorem 2.2]{GaoPestov2003}. Let $\bar{d}$ be the induced metric $\bar{d}(x, y) = \| x - y \|$. It's easy to check that if $X$ is separable, then so is $L(X)$, with countable dense set

\begin{equation}\label{eq:countable_dense}
\left\{\sum_{i=1}^n \lambda_i x_i \mid n \in \omega, \; \lambda_i \in \mathbb{Q}, \; x_i \in D\right\},  
\end{equation}
where $D$ is some countable dense subset of $X$.

The infimum in the definition of $\| x\|$ is always attained for some choice of $x_i, y_i$ with $x_i, y_i \in \supp(x)$ 
for every $1 \le i \le n$ (see \cite[Corollary 2.9]{GaoPestov2003}). 
From this fact it is easy to check that
\begin{equation}\label{eq:lower_bound}
\left\|\sum_{i=1}^n \lambda_i x_i \right\| \ge  \frac{1}{2} \left(\sum_{i=1}^n |\lambda_i|\right) \cdot \min \{d(x_i, x_j) \mid i \neq j\}
\end{equation}
always holds.

We can of course identify $X$ with its copy in $L(X)$ with the map $x \mapsto 1 x$. In fact, $\bar{d} \upharpoonright X = d$ (see \cite[Corollary 2.12]{GaoPestov2003} and the comments before it).

Define the free Banach space over $X$, denoted $B(X)$, to be the completion of $L(X)$.
The inclusion map $X \rightarrow B(X)$ is an isometric embedding.
Letting $Z^{<\omega} = \bigsqcup_{n < \omega} Z^n$ for any Polish $Z$, it is easy to check that the map
\[ (X \times \mathbb{R})^{<\omega} \rightarrow B(X), \quad (x_i, \lambda_i)_{i < n} \mapsto \sum_{i < n} \lambda_i \cdot x_i\]
is continuous.

\section{Countable treeable equivalence relations}\label{sec:treeable}

Our goal in this section is to see that every countable treeable Borel equivalence relation is classifiable by an abelian Polish group.
We first have to introduce a notion of Polish edge labelings, and a property of such labelings that we will call ``stretched".

\subsection{Polish edge labelings}

Given a Polish space $X$, a {\boldmath\bfseries (undirected) graph on $X$} is a set
\[G \subseteq (X \times X) \setminus \{(x, x) \mid x \in X\}\]
satisfying symmetry. We write $E_G \subseteq X \times X$ as the connectivity relation, which is Borel when $G$ is Borel and has countable components.

A {\boldmath\bfseries path in $G$} is a sequence $(x_0, x_1), (x_1, x_2), ..., (x_{n-1}, x_n)$ of edges of $G$. Say that two points $x$ and $y$ have {\boldmath\bfseries path distance $k$} iff $k$ is least such that there is a path consisting of $k$ edges from $x$ to $y$.

A {\boldmath\bfseries (directed) Polish edge labeling} $\gamma : \Lab \rightarrow G$ of a graph $G$ is a Polish space $\Lab$ and a continuous injective function $\gamma$ satisfying
\[\forall (x, y) \in G, \textrm{ exactly one of } (x, y) \textrm{ and } (y, x) \textrm{ is in the image of }\gamma.\]
In other words, the labeling associates to every edge a unique real and an orientation. It is clear by the usual change-of-topology arguments that every Borel graph has a Polish edge labeling. 

Suppose $\gamma : \Lab \rightarrow G$ is a Polish edge labeling of a graph living on $X$. Identify the image $\gamma[\Lab]$ with its copy in $B(\Lab)$ by the map 
\[(x, y) \mapsto 1 \cdot l, \quad \textrm{where }\gamma(l) = (x, y),\]
and $G \setminus \gamma[\Lab]$ with its copy in $B(\Lab)$ by the map 
\[(x, y) \mapsto -1 \cdot l, \quad \textrm{where }\gamma(l) = (y, x).\]

Think of $1 \cdot l$ as traveling ``forward" along the edge $\gamma(l)$, and $-1 \cdot l$ as travelling ``backwards" along $\gamma(l)$. Temporarily define $\sigma : \{(\pm 1) \cdot l \mid l \in \Lab\} \rightarrow G$ to be the inverse of the union of these two maps.

Say that a sum $\sum_{i=1}^n \lambda_i l_i \in B(\Lab)$ is a {\boldmath\bfseries path label} iff $\lambda_l = 1, -1$ for $i = 1, ..., n$, and $\sigma(\lambda_{j_1}l_{j_1}), ..., \sigma(\lambda_{j_n}l_{j_n})$ is a path in $G$ for some reordering $l_{j_1}, ..., l_{j_n}$ of $l_1, ..., l_n$. Note that such a path can only travel along an edge once. If the graph $G$ is acyclic, then there is exactly one path label between any two points in the same connected component. Naturally, we say that a path label starts at a vertex $x \in X$ and ends at a vertex $y \in X$, iff the path that it labels starts at $x$ and ends at $y$.

\begin{lemma}\label{lem:path_labels_borel}
If $\gamma : \Lab \rightarrow G$ is a Polish edge labeling and $G$ is acyclic, then the map
\[ E_G \rightarrow B(\Lab), \quad (x, y) \mapsto p_{x, y} \]
is Borel, where $p_{x, y}$ is the unique path label from $x$ to $y$.
\end{lemma}

\begin{proof}
Let $A \subseteq X^2 \times (\Lab \times \{-1, 1\})^{<\omega}$ be the Borel set of $(x, y, ((\ell_i, \lambda_i)_{i < n}))$ such that for some permutation $\pi$ of $\{0, ..., n-1\}$, we have that $x$ is the left-component of $\sigma(\lambda_{\pi(0)}  \ell_{\pi(0)})$, $y$ is the right-component of $\sigma(\lambda_{\pi(n-1)}  \ell_{\pi(n-1)})$, and for every $0 \le i < n-1$, the right component of $\sigma(\lambda_{\pi(i)} \ell_{\pi(i)})$ is equal to the left-component of $\sigma(\lambda_{\pi(i+1)} \ell_{\pi(i+1)})$.
Notice that for any $(x, y)$, the section $A_{(x, y)}$ must be finite.
In particular, by the remark at the end of Chapter 2, the set $B \subseteq X^2 \times B(\Lab)$ of all $(x, y, p_{x, y})$ is a continuous finite-to-one image of $A$.
The set $B$ is the graph of the desired function.
\end{proof}

We remark that path labels have the following nice closure property: if $p$ is a path label from $x$ to $y$ and $q$ is a path label from $y$ to $z$, then  $p+q$ is a path label from $x$ to $z$. In particular, if $G$ is acyclic, then $p_{x, z} = p_{x, y} + p_{y, z}$ for any $x, y, z$ in the same connected component.

\subsection{Stretched labelings}
We will want our edge labelings to satisfy a property that we will call ``stretched".

Let $G_{\Lab}$ be the {\boldmath\bfseries edge graph} on $\Lab$, with edge $(l_1, l_2)$ iff $\gamma(l_1)$ and $\gamma(l_2)$ are incident to the same vertex (regardless of direction).

\begin{definition}
Given a Polish edge labeling $\gamma : \Lab \rightarrow G$, and a compatible complete metric $d \le 1$ on $\Lab$, say that {\boldmath\bfseries $\gamma$ is stretched with respect to $d$} iff for some fixed $\epsilon > 0$, whenever $l_1$ and $l_2$ have path-distance $k \ge 1$ in the edge graph $G_{\Lab}$, we have $d(l_1, l_2) \ge \epsilon \cdot 1/k$.
\end{definition}

Locally-finite Borel graphs always have stretched Polish edge labelings:

\begin{lemma}\label{lem:stretched_labeling}
Let $G$ be a locally-finite Borel graph. Then there is a Polish edge labeling $\gamma : \Lab \rightarrow G$ and a compatible complete metric $d \le 1$ on $\Lab$ such that $\gamma$ is stretched with respect to $d$.
\end{lemma}

\begin{proof}
Fix a Polish edge labeling $\gamma : \Lab \rightarrow G$.
For every $n$, let $G_{\Lab}^{(n)}$ be the undirected graph with edge $(l_1, l_2)$ iff $l_1$ and $l_2$ have path-distance at most $2^n$ in the edge graph $G_{\Lab}$.
Notice that each $G_{\Lab}^{(n)}$ is Borel and locally-finite, so fix proper colorings $c^{(n)} : \Lab \rightarrow \omega$ of $G_{\Lab}^{(n)}$ for each $n$. 
Fix a compatible finer Polish topology on $\Lab$ making each $c^{(n)}$ continuous, and fix a compatible complete metric $d'$ on $\Lab$ inducing this topology.
We can ensure $d' \le 1$ by replacing $d'$ with the compatible complete metric defined by declaring the distance of any two $\ell_1, \ell_2 \in \Lab$ to be
\[ \frac{d'(\ell_1, \ell_2)}{d'(\ell_1, \ell_2) + 1} .\]
The map $x \mapsto (c^{(n)})_{n < \omega}$ is continuous from $\Lab$ to $\omega^\omega$. 
Take $\Lab'$ to be the (closed) graph of this function, and $\gamma'$ to be the projection onto the first coordinate, composed with $\gamma$. 
On $\Lab'$ we have the metric 
\[d((l_1, s_1), (l_2, s_2)) = \frac{1}{2}[d'(l_1, l_2) + d_{\omega^\omega}(s_1, s_2)],\]
where $d_{\omega^\omega}((a_n), (b_n)) = \sum \{1/{2^n} \mid a_n \neq b_n\}$ is the usual complete metric on $\omega^\omega$. 
It is easy to see $\gamma' : \Lab' \rightarrow G$ is stretched with respect to this metric and $d \le 1$.
\end{proof}

\begin{lemma}\label{lem:discrete_path_labels}
Suppose that $\gamma : \Lab \rightarrow G$ is a Polish labeling of a graph living on $X$, which is stretched with respect to a compatible complete metric $d \le 1$. Assume furthermore that $G$ is acyclic. Then for any $x \in X$, the set of path labels starting at $x$ is closed in $B(\Lab)$. In fact, they are uniformly discrete.
\end{lemma}

\begin{proof}
Fix some $\epsilon > 0$ from the stretchedness of the labeling.
Let $p_1$ be a path label from $x$ to $y$ and $p_2$ be a path label from $x$ to $y'$. Then $p_2 - p_1$ is a path label from $y$ to $y'$, and it does not re-use the same edge. Thus by Equation \ref{eq:lower_bound} we have
\[ \| p_2 - p_1 \| \ge (n/2) \cdot \min\{d(l_1, l_2) \mid l_1, l_2 \in \supp(p_2-p_1)\}\]
where $n$ is the length of the path $p_2-p_1$ (number of edges used). By stretchedness, we have $d(l_1, l_2) \ge \epsilon \cdot 1/(n-1)$ for every $l_1, l_2$ in the support of the path label $p_2-p_1$. Thus we have
\begin{align*}
    \|p_2 - p_1 \| &\ge (n/2) \cdot \min\{d(l_1, l_2) \mid l_1, l_2 \in \supp(p_2-p_1)\}\\
    &\ge (n/2) \cdot (\epsilon / (n-1)) \\
    &\ge \epsilon/2.
\end{align*}
This holds for any path labels $p_1$ and $p_2$ starting at $x$, and thus the set is uniformly discrete.
\end{proof}

\subsection{The reduction}
Suppose $G$ is a locally-finite acyclic graph on Polish $X$ with Polish edge labeling $\gamma : \Lab \rightarrow G$, and furthermore by Lemma \ref{lem:stretched_labeling} assume the labeling is stretched with respect to a compatible complete metric $d \le 1$ on $\Lab$.

Consider the map 
\[X \rightarrow \mathcal{F}(B(X) \times X), \quad x \mapsto \{(p_{x, y}, y) \mid y \mathrel{E_G} x\}.\]
By Lemma \ref{lem:discrete_path_labels}, this is well-defined.
To see that it is Borel, first observe that by Feldman-Moore there is a Borel function $f : X \rightarrow X^{\le \omega}$ such that $f(x)$ enumerates $[x]_{E_G}$. 
By Lemma \ref{lem:path_labels_borel}, there is a Borel function $g : X \rightarrow B(\Lab)^{\le \omega}$
such that $g(x)$ enumerates the set of paths starting at $x$. 
Finally, by Lemma \ref{lem:effros_borel}, the desired map is Borel.

Now consider the action
\[ a : B(\Lab) \curvearrowright \mathcal{F}\left(B(\Lab) \times X\right)\]
defined by
\[g \cdot C = \{(g + h, x) \mid (h, x) \in C\}.\]
It's easy to check that this is well-defined and Borel.

\begin{claim}
The map $f = x \mapsto \{(p_{x, y}, y) \mid y \mathrel{E_G} x\}$ is a reduction from $E_G$ to $E_a$.
\end{claim}

\begin{proof}
To see that it is a homomorphism, observe that if $x \mathrel{E_G} y$ and $p_{x, y}$ is the unique path from $x$ to $y$, then 
\[f(y) = \{(p_{y, z}, z) \mid z \mathrel{E_G} y\} = \{(p_{y, x} + p_{x, z}, z) \mid z \mathrel{E_G} y \mathrel{E_G} x\} = p_{y, x} + f(x).\]
On the other hand, if $f(x) \mathrel{E_a} f(y)$ then there must be some path label $p$ such that $(p, x) \in f(y)$, and thus $x \mathrel{E_G} y$.
\end{proof}

\begin{corollary}
Any countable treeable equivalence relation is classifiable by an abelian Polish group.
\end{corollary}

\begin{proof}
Any countable treeable equivalence relation is Borel-reducible to $E_G$ for some locally-countable Borel acyclic graph $G$ living on some Polish space $X$ (see \cite{JKL2002}). By Lemma \ref{lem:stretched_labeling}, we can find a Polish edge labeling $\gamma : \Lab \rightarrow G$ which is stretched with respect to a compatible complete metric $d \le 1$ on $\Lab$. As we just saw, we get that $E_G \le_B E_a$ for the described action.
\end{proof}

By \cite[Theorem 4.3]{GaoPestov2003}, every abelian Polish group is involved by $\ell_1$, so we recover the result from the title of this paper by Fact \ref{fact:mackey-hjorth} (Mackey-Hjorth).

\section{Actions of countable products of locally compact abelian Polish groups}\label{sec:lca_classifiable}

In this section we prove that countable Borel equivalence relations classifiable by $\R^\omega$ are hyperfinite.

We recall the following two results:

\begin{proposition}[Kechris \cite{Kechris1992}]\label{prop:kechris_ess_ctbl}
Suppose $G$ is a locally-compact Polish group and $a : G \curvearrowright X$ is a Borel action on a standard Borel space. Then $E_a$ is essentially countable.
\end{proposition}

\begin{proposition}[Cotton \cite{Cotton2019}]\label{prop:cotton_ess_ctbl}
Suppose $G$ is a locally-compact abelian Polish group and $a : G \curvearrowright X$ is a Borel action on a standard Borel space. Then $E_a$ is essentially hyperfinite.
\end{proposition}

Recall that an equivalence relation $E$ is essentially countable iff it is Borel reducible to a countable Borel equivalence relation, and essentially hyperfinite iff it is Borel reducible to a hyperfinite equivalence relation.

Given a continuous action $a : G \curvearrowright X$ for a Polish group $G$ and Polish space $X$, we recall the relation $\le_2$ on pairs $(x, U)$ for $x \in X$ and open $U \subseteq G$ from the Hjorth analysis of Polish group actions.
Write $(x, U) \le_2 (y, V)$ iff for every $U' \subseteq U$ open and nonempty, there is $V' \subseteq V$ open nonempty such that $V'y \subseteq \overline{U'y}$.
Finally, write $x \sim_1 y$ iff $(x, G) \le_2 (y, G)$ and $(y, G) \le_2 (x, G)$.
It is easy to check that $x \sim_1 y$ implies that $\overline{G \cdot x} = \overline{G \cdot y}$.

One can consult \cite{Drucker} or \cite{Allisontsi} for general references on the Hjorth analysis.
However, the next lemma isolates the particular fragment of the Hjorth analysis of Polish groups that we will need.

\begin{lemma}\label{lem:equiv}
Suppose $A, B \subseteq X$ are $G$-invariant $\BPi^0_3$ sets, and there exists $x \in A$ and $y \in B$ satisfying $x \sim_1 y$. Then $A$ and $B$ must intersect.
\end{lemma}

\begin{proof}

Let $\tau$ denote the original Polish $G$-space topology on $X$.
Suppose some $x, y \in X$ satisfy $(x, G) \le_2 (y, G)$ and $(y, G) \le_2 (x, G)$. 
Since $[x]_F$ and $[y]_F$ are both $\BPi^0_3$, we can write $[x]_F = \bigcap_n \bigcup_m A_{n, m}$ and $[y]_F = \bigcap_n \bigcup_m B_{n, m}$ where $A_{n, m}$ are $B_{n, m}$ are closed.

By basic properties of the Vaught transform (see \cite{BeckerKechris1996} for a definition and properties), we may write

\[
[x]_F = [x]_F^{*G}= \bigcap_n \Big[\bigcup_m A_{n, m}\Big]^{*G} = \bigcap_n \bigcap_V \Big[\bigcup_m  A_{n, m}\Big]^{\Delta V}= \bigcap_n \bigcap_V \bigcup_m  A_{n, m}^{\Delta V}
\]

and similarly
\[ [y]_F = \bigcap_n \bigcap_V \bigcup_m B_{n, m}^{\Delta V}\]
where $V$ ranges over basic open subsets of $G$.

By a sharp analysis due to Hjorth of Polish changes of topology (see \cite[Theorem 4.3.3]{Gao2008}), there is a finer Polish topology $\tau^*$ on $X$, generated by a countable family of sets that are either open in $\tau$ or of the form $C^{\Delta V}$ where $C$ is closed in $\tau$ and $V \subseteq G$ basic open, in which each $A_{n, m}^{\Delta V}$ is open.
In particular, $[x]_F$ and $[y]_F$ are $G_\delta$ in $\tau^*$. 
Thus we just have to show that $G \cdot x$ and $G \cdot y$ have the same closure in $\tau^*$, as this would mean that in the topology $\tau^*$, both $[x]_F \cap \overline{G \cdot x}$ and $[y]_F \cap \overline{G \cdot y}$ are dense $G_\delta$ subsets of $\overline{G \cdot x} = \overline{G \cdot y}$.
It is enough to show then that $x \in C^{\Delta G}$ iff $y \in C^{\Delta G}$ for any $\tau$-closed $C \subseteq X$.

Suppose first $x \in C^{\Delta G}$. 
Then there is some $g \in G$ and some symmetric open neighborhood $V$ of the identity of $G$ such that $g \cdot x \in C^{*V}$.
In fact, this means $\overline{Vg \cdot x} \subseteq C$.
Because $(x, G) \le_2 (y, G)$, there is some nonempty open $V' \subseteq G$ such that $\overline{V' \cdot y} \subseteq \overline{Vg \cdot x}$. 
In particular, 
\[V' \cdot y \subseteq \overline{Vg \cdot x} \subseteq C\]
and so $y \in C^{*V'}$ which implies $y \in C^{\Delta G}$.
The other direction is the same.
\end{proof}

We now use this fact to prove that any $\BPi^0_3$ orbit equivalence relation induced by an action of a countable product of Polish groups is in fact Borel reducible to a countable product of orbit equivalence relations, each induced by only finitely-many Polish groups from the sequence.

\begin{proposition}\label{prop:prodn_reduction}
Suppose $G = \prod_n G_n$ is a countable product of Polish groups and
\[a : \prod_n G_n \curvearrowright Z\]
is a continuous action on a Polish space $Z$.
Then there are continuous actions
\[b_n : \prod_{k \le n} G_k \curvearrowright Y_n\]
on Polish spaces $Y_n$ and a Borel function $g : Z \rightarrow \prod_n Y_n$ which is a Borel homomorphism from $E_a$ to $\prod_n E_{b_n}$ satisfying the additional property that for any $G$-invariant $\BPi^0_3$ sets $A, B \subseteq Z$ if there exist $x \in A$ and $y \in B$ satisfying $g(x) \mathrel{\prod_n E_{b_n}} g(y)$ then $A$ and $B$ must intersect.
\end{proposition}

\begin{proof}
For each $n$, let $H_{>n}$ be the closed normal subgroup of $G$ of sequences where $g_k = 1_{G_k}$ for each $k > n$ and $H_{\le n}$ be the closed normal subgroup of $G$ of sequences where $g_k = 1_{G_k}$ for each $k \le n$.

For each $n$, let $Y_n = \mathcal{F}(G)^{\omega}$ and let $b_n : H_{>n} \curvearrowright Y_n$ be the continuous action by right-translation, i.e.
\[ h \cdot \langle C_{k} \mid k \in \omega \rangle = \langle C_{k} h^{-1} \mid k \in \omega\rangle. \]
Let $U_k$ for $k \in \omega$ enumerate a basis for the topology on $Z$ and define $f_n : Z \rightarrow Y_n$ to be the Borel function
\[f_n(x) = \langle \overline{\{g \in G \mid H_{\le n}g \cdot x \cap U_k \neq \emptyset\}} \mid k \in \omega \rangle\]
This is easily seen to be a homomorphism from $E_a$ to $E_{b_n}$.

Thus the Borel map
\[ \bar{f} : Z \rightarrow \prod_n Y_n\]
defined by
\[ \bar{f}(x) = \langle f_n(x) \mid n \in \omega\rangle\]
is a Borel homomorphism from $E_a$ to $\prod_n E_{b_n}$.

Now let $A, B \subseteq X$ be $G$-invariant and $\BPi^0_3$ and suppose $x \in A$ and $y \in B$ satisfy $g(x) \mathrel{\prod_n E_{b_n}} g(y)$.
By Lemma \ref{lem:equiv} it is enough to see that $x \sim_1 y$.

Let $V \subseteq G$ be an arbitrary open set.
Let $n$ be large enough such that $V = VH_{\le n} = H_{\le n} V$.
By (3), there is some $\hat{h} \in H_{> n}$ such that for every $k$,
\[\overline{\{g \in G \mid H_{\le n}g \cdot y \cap U_k \neq \emptyset\}}\hat{h}^{-1} =\overline{\{g \in G \mid H_{\le n}g \cdot x \cap U_k \neq \emptyset\}}\]
or equivalently
\begin{equation}
\overline{\{g \in G \mid H_{\le n}g\hat{h} \cdot y \cap U_k \neq \emptyset\}} =\overline{\{g \in G \mid H_{\le n}g \cdot x \cap U_k \neq \emptyset\}}. \tag{*}
\end{equation}

We argue that $V\hat{h} \cdot y \subseteq \overline{V \cdot x}$.
Fix any $v \in V$ and fix an arbitrary basic open neighborhood $U_k$ of $v\hat{h} \cdot y$.
In particular, this means $v$ is in the left-hand side of $(*)$, and thus is in the right-hand side.
This means there is some $v' \in V$ such that $H_{\le_n}v' \cdot x \cap U_k \neq \emptyset$.
Because $H_{\le_n}v' \subseteq V$ and $v$ was chosen arbitrarily, we have that $V\cdot x \cap U_k \neq \emptyset$ as desired.
\end{proof}

Applying Propositions \ref{prop:kechris_ess_ctbl} and \ref{prop:cotton_ess_ctbl}, we get the following:

\begin{corollary}
If $G$ is a countable product of locally-compact Polish groups and $a : \prod_n G \curvearrowright X$
is a continuous action on a Polish space $X$ such that $E_a$ is $\BPi^0_3$, then $E_a \le_B E_\infty^\omega$. If $G$ is moreover a countable product of locally-compact abelian Polish groups, then $E_a \le_B E_0^\omega$.
\end{corollary}

We can now conclude the main result of the section.

\begin{theorem}
Let $E$ be a countable Borel equivalence relation that is classifiable by a Polish group $G$ which is a countable product of locally-compact abelian Polish groups. Then $E_a$ is hyperfinite.
\end{theorem}

\begin{proof}
Let $Z$ be a Polish space with continuous action $a : G \curvearrowright Z$ and suppose $E$ lives on a Polish space $X$.
Let $f : X \rightarrow Z$ be a Borel reduction from $E$ to $E_a$.
Let $B := f[X]$, which is Borel as $f$ is Borel and countable-to-one.
By a reflection argument \cite[Exercise 35.19]{Kechris1995}, let $F \supseteq E$ be a Borel equivalence relation satisfying
\[ \forall x, x' \in X, (f(x) \mathrel{F} f(y)) \rightarrow x \mathrel{E} y.\]
Let $P \subseteq Z \times X$ be defined by
\[ P := \{(z, x) \in Z \times X \mid z \mathrel{F} f(x)\}.\]
As this is a Borel set with countable sections, by Lusin-Novikov there is a Borel function $h : [B]_F \rightarrow X$ uniformizing $P$.
In particular, $h$ is a reduction from $F \upharpoonright B$ to $E$ satisfying $g(f(x)) \mathrel{E} x$ for every $x \in X$.

By \cite[Theorem 5.2.1]{BeckerKechris1996}, there is a finer Polish $G$-space topology on $Z$, compatible with the original topology in the sense that it generates the same Borel $\sigma$-algebra, in which $B$ is closed and $g^{-1}[U]^{\Delta V}$ is open for every basic open $U \subseteq Y$ and open $V \subseteq G$.

Let $g : Z \rightarrow Y$ be the function from Proposition \ref{prop:prodn_reduction} with $Y = \prod_n Y_n$.

\begin{claim}
The composition $g \circ f$ is a Borel reduction from $E$ to $\prod_n E_{b_n}$.
\end{claim}

\begin{proof}
Fix arbitrary $x,y \in X$.
First, suppose $x \mathrel{E} y$.
Then we have $f(x) \mathrel{E_a} f(y)$.
By the choice of $g$, we have $g(f(x)) \mathrel{\prod_n E_{b_n}} g(f(y))$ as desired.

Now suppose $g(f(x)) \mathrel{\prod_n E_{b_n}} g(f(y))$.
Because $h(f(x))$ is countable, it is $\Sigma^0_2$, and thus 
\[[h(f(x))]_E = \bigcup_n \bigcap_m (X \setminus U_{n, m})\]
for some sequence $(U_{n, m})$ of basic open subsets of $X$.
In particular,
\begin{align*}
    [f(x)]_F &= \Big[\bigcup_n \bigcap_m h^{-1}[X \setminus U_{n, m}])\Big] = \Big[\bigcup_n \bigcap_m h^{-1}[X \setminus U_{n, m}]\Big]^{\Delta G} \\
    &= \bigcup_n \Big[\bigcap_m h^{-1}[X \setminus U_{n, m}]\Big]^{\Delta G} = \bigcup_n \bigcup_V \Big[\bigcap_m h^{-1}[X \setminus U_{n, m}]\Big]^{*V} \\
    &= \bigcup_n \bigcup_V \bigcap_m h^{-1}[X \setminus U_{n, m}]^{*V} = \bigcup_n \bigcup_V \bigcap_m B \setminus h^{-1}[U_{n, m}]^{\Delta V}
\end{align*}
where $V$ ranges over basic open subsets of $G$.
In particular, $[f(x)]_F$ is $\BSigma^0_2$ and thus $\BPi^0_3$.
Similarly, so is $[f(y)]_F$.
Thus by choice of $g$ we have $f(x) \mathrel{F} f(y)$.
By choice of $F$ we have $x \mathrel{E} y$ as desired.
\end{proof}

By Proposition \ref{prop:cotton_ess_ctbl}, we have that $\prod_n E_{b_n}$ is Borel-reducible to $E_0^\omega$ and thus so is $E$.
By the Hjorth-Kechris Sixth Dichotomy Theorem (see \cite{HjorthKechris1997}), we have that $E$ is in fact hyperfinite.
\end{proof}

We remark that by Fact \ref{fact:mackey-hjorth} we could weaken the assumption in the previous corollary to say that $G$ is instead simply involved by a countable product of locally-compact abelian Polish groups.

We also note that the application of the Sixth Dichotomy Theorem could likely be eliminated by some finer application of the Hjorth analysis.

\section{Final remarks and further questions}

It is not clear that the use of treeability is necessary in Theorem \ref{thm:treeable}. It is still open if \emph{every} countable Borel equivalence relation is classifiable by an abelian Polish group. In fact, it is even open if classifiability by a tsi Polish group and by an abelian Polish group are the same. In \cite{AllisonPanagio}, it is shown that there is an equivalence relation classifiable by a cli Polish group (i.e. a Polish group with a compatible complete left-invariant metric) but not a tsi Polish group (i.e. a Polish group with a compatible complete two-sided metric). Thus the cli Polish groups have higher ``classification strength" than the tsi Polish groups. However it is not clear if the tsi Polish groups have higher classification strength than the abelian Polish groups.

The converse to Theorem \ref{thm:treeable} is false. It is obvious that the class of equivalence relations classifiable by abelian Polish groups is closed under products. However the class of countable Borel treeable equivalence relations is not. For example, if $F$ is the orbit equivalence relation induced by the free part of the action of the free group $F_2$ on $2^{F_2}$ by shifts, then $E \times E_0$ is classifiable by $G \times \mathbb{Z}$ where $G$ is an abelian Polish group which classifies $E$. However $E \times E_0$ is not treeable (see \cite[Corollary 3.28]{JKL2002}).

\bibliographystyle{alpha}
\bibliography{main}
\end{document}